\documentclass[11pt]{amsart}
\usepackage{amsmath}
\usepackage{amsfonts}
\usepackage{amssymb}
\usepackage[all,cmtip]{xy}           

\usepackage{dsfont}

\usepackage{bbm}
\usepackage{bbding}
\usepackage{txfonts}
\usepackage[shortlabels]{enumitem}
\usepackage{ifpdf}
\ifpdf
\usepackage[colorlinks,linkcolor=black,final,backref=page,hyperindex]{hyperref}
\else
\usepackage[colorlinks,linkcolor=black,final,backref=page,hyperindex,hypertex]{hyperref}
\fi
\usepackage{amscd}
\usepackage{tikz-cd}
\usepackage[active]{srcltx}
\usepackage{mathrsfs}
\usepackage{ulem}
\makeatletter
\newtheorem{theorem}{Theorem}[section]
\newtheorem{proposition}[theorem]{Proposition}

\theoremstyle{definition} 
\newtheorem{example}[theorem]{Example}
\newtheorem{definition}[theorem]{Definition}
\newtheorem{remark}[theorem]{Remark}

\newcommand{\vv}{\mathsf{v}}
\newcommand{\ww}{\mathsf{w}}

\newcommand{\qa}{kQ/I}

\newcommand{\rad}{\mathrm{rad}}

\begin{document}
	
	\title[Invariants of derived equivalences for admissible fractional Brauer graph algebras]{
	Invariants of derived equivalences for admissible fractional Brauer graph algebras}

	\author{Bohan Xing} 
	\address{(Bohan Xing)
		School of Mathematical Sciences,
		Laboratory of Mathematics and Complex Systems,
		Beijing Normal University,
		Beijing 100875,
		P.R.China}
	\email{bhxing@mail.bnu.edu.cn}

	\date{\today}

	\begin{abstract}
	Characterizing derived equivalences between algebras via combinatorial structures has recently become a popular topic. In this paper, we study admissible fractional Brauer graph algebras, a new subclass of self-injective special biserial algebras, and provide several easily checkable combinatorial invariants for derived equivalences between them. In particular, we show that these algebras can be viewed as repetitive algebras and $r$-fold trivial extensions of gentle algebras.
	\end{abstract}

	\renewcommand{\thefootnote}{\alph{footnote}}
\setcounter{footnote}{-1} 
\renewcommand{\thefootnote}{\alph{footnote}}
\setcounter{footnote}{-1} 
	\maketitle
	

	\allowdisplaybreaks
	
	\section{Introduction}

	Derived equivalence is a classical topic in the representation theory of algebras. Two derived equivalent algebras share similar homological properties. By \cite{Ric,Kel}, a fundamental method for determining whether two algebras are derived equivalent is to find a suitable tilting complex and compute its endomorphism algebra. Although this approach always works in theory, in practice it often involves extensive calculations. Therefore, we generally hope that within special classes of algebras, derived equivalence can be characterized by certain complete invariants that are easy to verify.

In recent years, the study of certain classical finite-dimensional algebras using combinatorial structures and geometric models has become increasingly popular. For these algebras, derived equivalence can often be completely described by combinatorial or geometric data. For example, derived equivalence of gentle algebras can be determined by their geometric models \cite{AG,APS,JSW}, and derived equivalence of Brauer graph algebras (abbr. BGAs) can be determined by their ribbon graphs and multiplicity functions \cite{A,AZ,OZ}.

	Although many deep results are known for well-structured symmetric algebras such as BGAs, analogous characterizations have long been unavailable for self-injective algebras in general. 
	
	In order to study self-injective algebras via a similar combinatorial approach, \cite{LL} generalized BGAs to fractional BGAs. In general, fractional BGAs are neither special biserial nor of tame representation type \cite{X}. Consequently, in \cite{LL3}, the authors considered a subclass, namely fractional BGAs of type MS on which the Nakayama automorphism acts admissibly; these algebras are self-injective and special biserial. Here, ``MS'' stands for ``multiserial and self-injective'' \cite{GS1}. In this article, for simplicity, we refer to such algebras as {\it admissible fractional BGAs} (abbr. AFBGAs).
As in the case of BGAs, AFBGs are defined from combinatorial data --- a ribbon graph $\Gamma$ equipped with a (fractional) multiplicity function $m=\frac{d}{val}$ (see Definition \ref{fms-BG} for details) --- called an admissible fractional Brauer graph (abbr. AFBG). 

In fact, we will prove in Proposition \ref{prop:determined-by-BG} that AFBGAs are uniquely determined by their combinatorial data. We illustrate these concepts with the following example.

	\begin{example}\label{exa:preproj-A3}
		Consider the following ribbon graph $\Gamma$, consisting of two vertices $\vv$ and $\ww$ and three edges $1,2,3$, embedded in a torus with two punctures. Without loss of generality, we assume that the cyclic ordering around each vertex is clockwise.

\begin{center}

\tikzset{every picture/.style={line width=0.75pt}} 

\begin{tikzpicture}[x=0.75pt,y=0.75pt,yscale=-1,xscale=1]

\draw  [fill={rgb, 255:red, 0; green, 0; blue, 0 }  ,fill opacity=1 ] (80,95) .. controls (80,92.24) and (82.24,90) .. (85,90) .. controls (87.76,90) and (90,92.24) .. (90,95) .. controls (90,97.76) and (87.76,100) .. (85,100) .. controls (82.24,100) and (80,97.76) .. (80,95) -- cycle ;
\draw  [fill={rgb, 255:red, 0; green, 0; blue, 0 }  ,fill opacity=1 ] (180,95) .. controls (180,92.24) and (182.24,90) .. (185,90) .. controls (187.76,90) and (190,92.24) .. (190,95) .. controls (190,97.76) and (187.76,100) .. (185,100) .. controls (182.24,100) and (180,97.76) .. (180,95) -- cycle ;
\draw [line width=1.5]    (85,95) -- (185,95) ;
\draw [line width=1.5]    (85,95) .. controls (150.25,60.25) and (230.25,61.25) .. (185,95) ;
\draw [line width=1.5]    (85,95) .. controls (39.75,60.25) and (119.75,60.25) .. (185,95) ;

\draw (65.25,90) node [anchor=north west][inner sep=0.75pt]   [align=left] {$\vv$};
\draw (194.25,90) node [anchor=north west][inner sep=0.75pt]   [align=left] {$\ww$};
\draw (75.25,55) node [anchor=north west][inner sep=0.75pt]   [align=left] {$1$};
\draw (186.25,55) node [anchor=north west][inner sep=0.75pt]   [align=left] {$2$};
\draw (128.75,99.5) node [anchor=north west][inner sep=0.75pt]   [align=left] {$3$};
\end{tikzpicture}

	\end{center}
	As in the case of BGAs, the above ribbon graph $\Gamma$ naturally induces a quiver $Q_\Gamma$ as follows, whose vertices correspond to the edges of $\Gamma$, and whose arrows are determined by the cyclic ordering around the vertices of $\Gamma$.
	$$
\begin{tikzcd}
                                                                         & 2 \arrow[rd, "\alpha_2", shift left] \arrow[rd, "\beta_2"', shift right] &                                                                                                \\
1 \arrow[ru, "\alpha_1", shift left] \arrow[ru, "\beta_1"', shift right] &                                                                          & 3 \arrow[ll, "\beta_3"', bend left, shift right] \arrow[ll, "\alpha_2", bend left, shift left]
\end{tikzcd}$$
Here $\alpha$ denotes the arrows induced by the cyclic ordering around the vertex $\vv$, while $\beta$ denotes the arrows induced by the cyclic ordering around the vertex $\ww$.

In contrast to the multiplicity function for BGAs, which specifies how many times a path may wind around a given vertex, \cite{LL} define a {\it degree function} in the setting of AFBGAs, which records the maximal number of steps around a vertex, or equivalently, the maximal length of such a path.
Therefore, for a BGA, the degree $d(\vv)$ of each vertex $\vv$ in the ribbon graph is given by the product $d(\vv)=m(\vv)\,val(\vv)$, where $val(\vv)$ denotes the valency of $\vv$. Thus, one may regard $(\Gamma,d)$ as defining a BGA provided that the degree function $d$ is divisible by the valency function $val$.

{\it Thus, to avoid confusion with the multiplicity function $m$, throughout this paper we consider a pair $(\Gamma,d)$, where $d$ is understood to be the degree function unless otherwise specified.}

For instance, If we consider the degree function $d$ on $\Gamma$ defined by
$$
d(\vv)=d(\ww)=2,
$$
which corresponds to the multiplicity function $m$ on $\Gamma$ given by
$$
m(\vv)=m(\ww)=\frac{2}{3},
$$
then $(\Gamma,d)$ forms a AFBG, and the associated algebra admits the following commutative relations:
$$
\alpha_2\alpha_1=\beta_2\beta_1,
\quad
\alpha_3\alpha_2=\beta_3\beta_2,
\quad
\alpha_1\alpha_3=\beta_1\beta_3.
$$
As in the case of BGAs, we further impose ``gentle relations'' of the form $$\alpha\beta
=
\beta\alpha
=
0.$$ When the degree function $d$ satisfies certain additional conditions as in Definition~\ref{fms-BG} (ensuring that the paths appearing in the above relations have the same source and terminus), the resulting algebra $A$ associated with $(\Gamma,d)$ is a finite-dimensional self-injective special biserial algebra. The indecomposable projective $A$-modules are given by
			$$P_1=\begin{array}{*{1}{lll}}
			&1&\\2&&2\\&3&
		\end{array},\quad P_2=\begin{array}{*{1}{lll}}
			&2&\\3&&3\\&1&
		\end{array},P_3=\begin{array}{*{1}{lll}}
			&3&\\1&&1\\&2&
		\end{array}.$$
The Nakayama automorphism $\nu_A$ of the self-injective algebra $A$, given by
$$
e_i \mapsto e_{i+1},\quad
\alpha_i \mapsto \alpha_{i+1},\quad
\beta_i \mapsto \beta_{i+1},
\qquad i\in \mathbb{Z}/3\mathbb{Z}.
$$
is naturally corresponding to the following automorphism of $\Gamma$:
$$
1\mapsto 2,\quad
2\mapsto 3,\quad
3\mapsto 1,
$$
which is called the {\it Nakayama automorphism} $\nu$ of $(\Gamma,d)$. See Subsection \ref{subsec:def-fms-BGA} for details. 
	\end{example}

Similar to the case of gentle algebras \cite{SZ} and BGAs \cite{AZ}, AFBGAs are expected to be closed under derived equivalence (which is left as a conjecture in \cite{X2}). In particular, the representation-finite part \cite{LL3} and the tilting-discrete part \cite{X2} of such algebras are closed under derived equivalence.

Therefore, in this paper, we aim to provide some easily checkable invariants for derived equivalences within this class of algebras with combinatorial structures. The proofs of these invariants differ from the classical case of BGAs in \cite{A,AZ}, because the existence of a non-trivial Nakayama automorphism makes the classical arguments more complicated (in some cases, they even fail to work). Since in \cite{LL3} each AFBGA $A$ is associated with a BGA $A_{\mathrm{red}}$ (called the {\it reduced form} of $A$), our strategy is to use covering techniques on derived categories \cite{Asa}, which were used in \cite{Asa1} to classify representation-finite self-injective algebras. This allows us to construct derived equivalences in the BGA setting and, consequently, obtain combinatorial invariants as follows (since derived equivalence coincides with Morita equivalence for local algebras \cite{RZ}, we may restrict our discussion to the non-local case).

\begin{theorem}[see Theorem~\ref{thm:derived-eqv-between-BGAs}]
Let $A$ and $B$ be non-local AFBGAs associated with AFBGs $(\Gamma, d)$ and $(\Gamma', d')$, respectively. If $A$ and $B$ are derived equivalent, then the following statements hold:
\begin{enumerate}
\item The associated BGAs $A_{\mathrm{red}}$ and $B_{\mathrm{red}}$ are derived equivalent;
\item The following conditions are satisfied:
\begin{itemize}
\item $\Gamma$ and $\Gamma'$ have the same number of vertices and edges;
\item $(\Gamma,d)$ and $(\Gamma',d')$ have the same multisets of vertex multiplicities;
\item either both $\Gamma$ and $\Gamma'$ are bipartite, or neither is.
\end{itemize}
\end{enumerate}
\end{theorem}
\noindent
Note that in Remark \ref{rmk:1} we point out that the above invariants are not complete, and in Remark \ref{rmk:2} we propose a potential complete invariant for derived equivalence.

	It is often observed that many results for gentle algebras also extend to BGAs, mainly because they are related via trivial extensions \cite{Rin2,Sch}. In this paper, we also show that AFBGAs exhibit a similar connection with gentle algebras. Indeed, the relationships among AFBGAs, BGAs, and gentle algebras can be illustrated by the following diagram.

	
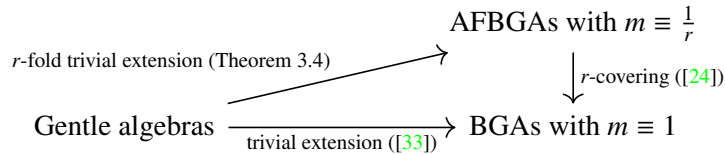
\begin{center}
\begin{tikzcd}
    &  &  & \text{AFBGAs with $m\equiv \frac{1}{r}$} \arrow[d, "\text{$r$-covering (\cite{LL3})}"] \\
\text{Gentle algebras} \arrow[rrru, "\text{$r$-fold trivial extension (Theorem \ref{thm:r-fold-tri})}"] 
\arrow[rrr, "\text{trivial extension (\cite{Sch})}"'] 
&  &  & \text{BGAs with $m\equiv 1$}
\end{tikzcd}
	\end{center}

	\medskip
	\textbf{Outline.}\; In Section \ref{sec:BGA}, we recall some classical results on AFBGAs and prove that, up to isomorphism, AFBGAs are uniquely determined by their AFBGs. In Section \ref{sec:3}, we construct the repetitive algebras and the $r$-fold trivial extensions of gentle algebras, and show that they are all AFBGAs, that is, they can be described by ribbon graphs together with multiplicity functions. In Section \ref{sec:4}, we present invariants for derived equivalences between AFBGAs.

	\section*{Acknowledgments}
	The author sincerely thanks Aaron Chan for his supervision at Nagoya University, many fruitful discussions and comments on this work. The author also thanks Pengyun Chen, Nengqun Li and Yuming Liu for helpful discussions. This work is supported by the China Scholarship Council (No. 202506040127).

	\section*{Notation}
	
		Throughout we assume that $k$ is an algebraically closed field and all algebras considered are $k$-algebras. Unless stated otherwise, all modules will be finitely generated left modules. Furthermore, we say that $A$ is a bound quiver algebra, if $A\cong \qa$, where $Q$ is a finite quiver and $I$ is an admissible ideal in the path algebra $kQ$. We denote by $s(p)$ the source vertex of a path $p$ and by $t(p)$ its terminus vertex. We will write paths from right to left, for example, $p=\alpha_{n}\alpha_{n-1}\cdots\alpha_{1}$ is a path with starting arrow $\alpha_{1}$ and ending arrow $\alpha_{n}$. A path is called a cycle if $s(p)=t(p)$.
	By abuse of notation we sometimes view an element in $kQ$ as an element in the quotient $\qa$ if no confusion can arise.

	In this paper, we study the indecomposable $A$-modules $M$ via their {\it Loewy structure}, which is represented by a diagram where the $i$-th row corresponds to the simple summands of the completely reducible module $\rad^{i-1}(A)M/\rad^{i}(A)M$ with $\rad(A)$ the Jacobson radical of $A$. Each number in the diagram denotes a distinct simple module in $A$. For further details, see for example in \cite[Page 174]{Ben}.

	Recall that for an algebra $A$, the {\it derived category} $\mathcal{D}(A)$ is the localization of homotopy category $\mathcal{K}(A)$ by inverting quasi-isomorphisms. We say that two algebras are {\it derived equivalent} if their derived  categories are equivalent as triangulated categories.

	A $k$-algebra $A$ is called {\it self-injective} if $A_A$ is an injective $A$-module; and {\it symmetric} if $_AA_A\cong DA$ as $A$-$A$-bimodules. Indeed, symmetric algebras are naturally self-injective (see for example in \cite{Z}).
	
	A $k$-algebra $A$ is {\it special biserial} if it is isomorphic to an algebra of the form $\qa$ where $kQ$ is a path algebra and $I$ is an admissible ideal such that the following properties hold.
	
	\begin{enumerate}[(1)]
		\item At every vertex $i$ in $Q$, there are at most two arrows starting at $i$ and there are at most two arrows ending at $i$.
		
		\item For every arrow $\alpha$ in $Q$, there exists at most one arrow $\beta$ such that $\beta\alpha\notin I$ and there exists at most one arrow $\gamma$ such that $\alpha\gamma\notin I$.
	\end{enumerate}

	\section{Admissible fractional Brauer graph algebras}\label{sec:BGA}
		
	\subsection{Ribbon graphs}
	\
	
	Ribbon graphs combinatorially encode the structure of oriented surfaces with boundary (see for example in \cite[Section 1.1]{OZ}). A key feature of ribbon graphs is the cyclic ordering of (half-)edges at each vertex, which captures the orientation data of the underlying surface. We begin this section by recalling their formal definition.
	
	\begin{definition}\label{def:ribbon-graph}
		A {\it ribbon graph} is a tuple $\Gamma=(V,H,s,\iota,\rho)$, where
		\begin{enumerate}[(1)]
			\item $V$ (also denoted by $V(\Gamma)$) is a finite set whose elements are called vertices;
			
			\item $H$ (also denoted by $E(\Gamma)$) is a set whose elements are called half-edges;
			
			\item $s: H\rightarrow V$ is a functions;
			
			\item $\iota: H\rightarrow H$ is an involution without fixed points;
			
			\item $\rho: H\rightarrow H$ is a permutation whose cycles correspond to the sets $H_v:=s^{-1}(v)$, $v\in V$.
			
		\end{enumerate}
	The $\iota$-orbits are called the edges of $\Gamma$. In particular, if we set
\[
E(\Gamma) := H/\langle \iota \rangle,
\]
then $(V(\Gamma), E(\Gamma))$ is the underlying combinatorial graph of $\Gamma$.
	\end{definition}
	
	For a ribbon graph $\Gamma=(V,H,s,\iota,\rho)$, we introduce the following notation. 
For each half-edge $h \in H$, we write
\[
h^+ := \rho(h) \quad \text{and} \quad h^- := \rho^{-1}(h)
\]
for the successor and predecessor of $h$, respectively. We denote by
\[
\bar{h} := \{h, \iota(h)\}
\]
the edge associated with $h$ in the underlying combinatorial graph of $\Gamma$.
For each vertex $\vv \in V$, the valency of $\vv$ is defined by
\[
val(\vv) := \bigl|\{h \in H \mid s(h) = \vv\}\bigr|.
\]
In particular, a loop (that is, an edge $\{h, \iota(h)\}$ with $s(h) = s(\iota(h))$) contributes twice to $val(\vv)$.

	Unless stated otherwise, we will assume that $\Gamma$ is connected, i.e. its underlying graph is connected.
	
	\begin{remark}
		Because we will later consider repetitive algebras, which have infinitely many simple modules, we assume in the definition of ribbon graphs that the set of half-edges $H$ is allowed to be infinite.
	\end{remark}

	\subsection{Admissible fractional Brauer graph algebras}\label{subsec:def-fms-BGA}
	\
	
	In this section, we review some basic knowledge about admissible fractional Brauer graph algebras, which is constructed in \cite{LL}. 
	
	\begin{definition}\textnormal{(cf. \cite{LL3, X2})}\label{fms-BG}
		An \textit{admissible fractional Brauer graph} (abbr. AFBG) is a pair $(\Gamma,d)$ consisting of a ribbon graph $\Gamma$ together with a {\it degree function} $d : V(\Gamma) \to \mathbb{Z}_{>0}$ (whose values are referred to as the {\it degrees}), such that for each half-edge $h \in H(\Gamma)$,
		\begin{itemize}
		\item  $\iota\!\left(\rho^{d(s(h))}(h)\right) \;=\; \rho^{d(s(\iota(h)))}(\iota(h))$;
		\item $\rho^{nd(s(h))}(h)\neq \iota(h)$, for all $n\in \mathbb{Z}$.
		\end{itemize}
	\end{definition}
	We often omit $d$ from the notation and simply refer to $\Gamma$ as an AFBG when no confusion arises. For a vertex $\vv$, we define the {\it multiplicity} of $\vv$ (also called the fractional degree in \cite{LL}) to be the rational number
\[
m(\vv) := \frac{d(\vv)}{val(\vv)}.
\]
It follows immediately from the definition that an AFBG is a (classical) Brauer graph (abbr. BG) (for an explicit definition, see for example in \cite{SS,OZ}) if and only if the multiplicity of every vertex is an integer. A vertex $\vv$ is called {\it truncated} if $d(\vv)=1$.
	Denote by $\nu$ the map
\begin{equation*}
\nu: H \longrightarrow H, \qquad
h \longmapsto \rho^{d(s(h))}h,
\end{equation*}
which is called the {\it Nakayama automorphism} of an AFBG $(\Gamma,d)$.
	
	To each AFBG $(\Gamma,d)$, one can associate a ($2$-regular) quiver $Q=Q_\Gamma$ and an ideal of relations $I=I_{\Gamma,d}$ in the path algebra $kQ$ as follows.
	
	\begin{enumerate}[(1)]
		\item The vertices of $Q$ correspond to the edges of $\Gamma$ and for every $h\in H$, there is an arrow $\alpha_h:\bar{h}\rightarrow\overline{h^+}$. The assignment $\alpha_h\mapsto \alpha_{h^+}$ defines a permutation $\rho=\rho_\Gamma$ of the arrows of $Q$ whose orbits are in bijection with vertices in $\Gamma$. Hence every arrow $\alpha$ defines a cycle
		$$C_\alpha=\alpha\rho(\alpha)\cdots\rho^l(\alpha)$$
		where $l+1$ denotes the cardinality of the $\pi$-orbit of $\alpha$. Every vertex of $Q$ is the starting point of at most two cycles of form $C_\alpha$. If $\alpha=\alpha_h$, set $d(\alpha):=d(s(h))$.
		
		\item The ideal $I$ is generated by the following set of relations:
		\begin{enumerate}[(i)]
			\item $$\alpha\rho(\alpha)\cdots\rho^{d(\alpha)-1}(\alpha)=\beta\rho(\beta)\cdots\rho^{d(\beta)-1}(\beta),$$
			where $\alpha,\beta\in Q_1$ and $s(\alpha)=s(\beta)$, that is $\alpha$ and $\beta$ start at the same edge of $\Gamma$.
			
			\item $$\alpha\beta=0,$$
			where $\alpha,\beta\in Q_1$ are composable and $\rho(\alpha)\neq \beta$;
		\end{enumerate}
	\end{enumerate}
	
	\begin{definition}\textnormal{(cf. \cite[Definition 2.2 and Proposition 5.5]{LL3})}\label{fms-BGA}
		A $k$-algebra $A$ is called a {\it admissible fractional Brauer graph algebra} (abbr. AFBGA) if there exists an AFBG $(\Gamma,d)$ such that $A\cong kQ_\Gamma/I_{\Gamma,d}$ as $k$-algebras.
	\end{definition}
	
	Note that $A$ is indecomposable as a ring if and only if $\Gamma$ is connected, and $A$ is a (classical) Brauer graph algebra (abbr. BGA) if and only if $m(\vv)$ is an integer for every vertex $\vv \in V(\Gamma)$. In fact, BGAs are coincide with symmetric special biserial algebras \cite{Sch}. Moreover, AFBGAs satisfy the following elementary properties; see~\cite{LL} if one needs details.

	\begin{proposition}\label{prop:bfBGA-basic}
Let $A$ be an AFBGA with associated AFBG $(\Gamma,d)$, and let $\nu$ be the Nakayama permutation of $(\Gamma,d)$. Then $A$ is self-injective and special biserial.
Moreover, the Nakayama automorphism of $A$ is given by the map
\begin{equation*}
\nu_A \colon A \to A,\qquad 
\nu_A(\bar{h})=\overline{\rho^{-d(s(h))}(h)}, \quad 
\nu_A(\alpha_h)=\rho^{-d(s(h))}(\alpha_h),
\end{equation*}
which is induced by the Nakayama automorphism of $(\Gamma,d)$. 
Here $\bar{h}$ is identified with the primitive idempotent $e_{\bar{h}}$.
\end{proposition}

\subsection{Reduced forms}\label{subsec:reduced-form}
\

For each AFBG $(\Gamma = (V, H, s, \iota, \rho), d)$, define its {\it reduced form} to be
\[
(\Gamma_{\mathrm{red}},d)=(\Gamma / \langle \nu \rangle,d) := ((V, H', s', \iota', \rho'),d),
\]
where $V$ coincides with the same vertex set of $\Gamma$ and $d$ is the same function defined on $V$, and 
\begin{enumerate}
    \item $H' = H / \langle \nu \rangle = \{ [h] \mid h \in H \}$ is the set of $\langle \nu \rangle$-orbits in $H$;
    \item $s'([h]) = s(h)$;
    \item $\iota'([h]) = [\iota(h)]$;
    \item $\rho'([h]) = [\rho(h)]$.
\end{enumerate}
In fact, $(\Gamma_{\mathrm{red}},d)$ is a Brauer graph (see for example \cite[Section 2]{LL3}). 

Let $A$ denote the AFBGA associated with $(\Gamma,d)$. We define $A_{\mathrm{red}}$ to be the BGA associated with $(\Gamma_{\mathrm{red}},d)$, which we also call the {\it reduced form} of $A$.

We note that the reduced form of the algebra $A$ in Example \ref{exa:preproj-A3} is the Brauer graph algebra associated with the ribbon graph consisting of a single edge connecting two distinct vertices, both of multiplicity two.

For the following statement, we recall that a Brauer graph is called a {\it Brauer tree} if its associated ribbon graph $\Gamma$ is a tree and the multiplicity function $m$ assigns the value $1$ to all vertices, except possibly for a single vertex $\vv$, called the {\it exceptional vertex}, for which $m(\vv)$ is referred to as the {\it exceptional multiplicity}. The BGA associated with a Brauer tree is called a {\it Brauer tree algebra}. In fact, a BGA is representation-finite if and only if it is a Brauer tree algebra (see for example \cite{SS}).

It is shown in~\cite[Theorem~2.29]{LL3} that an AFBGA is representation-finite (resp.\ domestic) if and only if its reduced form is representation-finite (resp.\ domestic). Moreover, in the representation-finite case, we have the following theorem.

\begin{theorem}\textnormal{(cf. \cite[Theorem 4.8]{LL3})}\label{thm:rep-fin-fms}
Let $A$ be an AFBGA with associated AFBG $(\Gamma,d)$. Then $A$ is representation-finite if and only if its reduced form $A_{\mathrm{red}}$ is a Brauer tree algebra with $n$ edges and exceptional multiplicity $m$. Equivalently, $\Gamma / \langle \nu \rangle$ is a Brauer tree. In this case, the stable Auslander--Reiten quiver of $A$ is isomorphic to $\mathbb{Z}A_{mn}/\langle \tau^{nr}\rangle$ for some positive integer $r$, where $\tau$ denotes the translation of $\mathbb{Z}A_{mn}$. That is, $A$ is derived equivalent to a self-injective Nakayama algebra.
\end{theorem}

\subsection{Determination by admissible fractional Brauer graphs}
\

In this section, we prove that, similarly to the case of BGA \cite[Lemma 3.1]{AZ}, an AFBGA is uniquely determined by its associated AFBG.

	\begin{proposition}\label{prop:determined-by-BG}
Let $A$ be an indecomposable self-injective special biserial $k$-algebra admitting a presentation $(Q,I)$ of the form given in the definition of AFBGAs, and let $(\Gamma,d)$ be the associated AFBG. If $(\Gamma,d)$ is not one of the exceptional BGs described in \cite[Lemma 3.1]{AZ}, then the underlying graph $\Gamma$ is independent of the choice of the presentation $(Q,I)$.
\end{proposition}

	\begin{proof}
		While the structure of projective $A$-modules can be directly deduced from $(\Gamma,d)$, we aim to show that the AFBG $(\Gamma,d)$ itself can be reconstructed from the module category of $A$, independently of the specific presentation $kQ/I$. Without loss of generality, we may assume that $A$ is not radical square zero. Otherwise, $A$ is a self-injective Nakayama algebra, and all its indecomposable projective modules are uniserial. In this case, its AFBG $(\Gamma,d)$ is uniquely determined by the orbits of simple modules under the Nakayama automorphism.
		
		Suppose that $A$ admits two presentations $kQ/I \cong kQ'/I'$, both satisfying the definition of $f_{ms}$-BGAs. This induces a Morita equivalence between $kQ/I\text{-}\mathsf{mod}$ and $kQ'/I'\text{-}\mathsf{mod}$. Since the ideals $I$ and $I'$ can be chosen to be admissible, we may identify $Q$ and $Q'$. Consequently, we obtain a bijection between:
    \begin{itemize}
        \item Primitive idempotents of the two presentations,
        \item Simple modules \(S_i\) over \(kQ/I\) and \(kQ^{\prime}/I^{\prime}\),
        \item Edges of the ribbon graph \(\Gamma\) and \(\Gamma^{\prime}\).
    \end{itemize}

    We now reconstruct \(\Gamma\) and \(\Gamma^{\prime}\) from the module categories. It suffices to show that, for each vertex of $\Gamma$, the cyclic order around it is uniquely determined by the corresponding module category. For each projective cover \(P_i\) of \(S_i\), consider the module
    $\mathrm{rad}\,P_i / \mathrm{soc}\,P_i$.
    Since $A$ is special biserial, hence biserial, this module decomposes into at most two uniserial summands \(M_i\) and \(N_i\). Each summand yields a sequence of simple modules via its radical series, for instance, \((S_{i_1}, \dots, S_{i_m})\) for \(M_i\). We continue extending this sequence via the radical series of the projective cover of $S_{i_1}$, and append $S_i$ to obtain a {\it primitive} cyclic sequence, that is, a cyclic sequence which is not a nontrivial power of a shorter cyclic sequence.:
    \[
    (S_i = S_{i_0}, S_{i_1}, \dots, S_{i_n}),
    \]
    indexed by \(\mathbb{Z}/(n+1)\mathbb{Z}\), considered up to cyclic permutation. Let \(S_i\) be a simple module with associated cyclic sequence(s).

    \begin{enumerate}
        \item \textbf{\(S_i\) appears in two distinct cyclic sequences.} Then the corresponding edge is not a loop. The cyclic orderings at its endpoints are determined by subsequences of the form \((S_i, S_{i_1}, \dots, S_{i_l}, S_i)\) excluding \(S_i\) internally. Vertex degrees equal the length of the sequence \((S_{i_1}, \dots, S_{i_m})\) plus $1$ for each uniserial summand $M_i$.

        \item \textbf{\(S_i\) appears in one cyclic sequence with a subsequence \(\sigma = (S_i, S_{i_1}, \dots, S_{i_l}, S_i, S_{i_{l+2}}, \dots, S_{i_m})\) where \((S_{i_1}, \dots, S_{i_l})\) and \((S_{i_{l+2}}, \dots, S_{i_m})\) are distinct and do not contain \(S_i\).} 
               Then the edge is a loop. The cyclic orderings at its endpoints are determined by $\sigma$. Vertex degrees also equal the length of the sequence \((S_{i_1}, \dots, S_{i_m})\) plus $1$ for each uniserial summand $M_i$.

        \item \textbf{\(S_i\) appears in one cyclic sequence without such \(\sigma\), and \(P_i\) is uniserial.} 
            We reconstruct a vertex $\vv$ with degree $1$ and each half-edge in the same $\langle \nu\rangle$-orbit will connect with this vertex. the permutation induced by $\nu^{-1}$ will also give the cyclic ordering around $\vv$.

        \item \textbf{\(S_i\) appears in one cyclic sequence without \(\sigma\), but \(P_i\) is non-uniserial.} \\
              In this case, $\Gamma$ is determined as a caterpillar as shown in \cite[Section 3]{AZ} and vertex degrees also equal the length of the sequence \((S_{i_1}, \dots, S_{i_m})\) plus $1$ for each uniserial summand $M_i$. In this case, $\Gamma$ may be either a caterpillar with one vertex or with two distinct vertices, where each vertex has the same degree. By Theorem \ref{thm:derived-eqv-between-BGAs}, they are not even derived equivalent, and hence in particular not isomorphic. Therefore, the above construction uniquely determines a ribbon graph $\Gamma$.
    \end{enumerate}

	We note that the exceptional cases in \cite[Lemma 3.1]{AZ}, namely the BGAs $k\langle x,y\rangle/(xy-yx,x^2,y^2)$ and $k\langle x,y\rangle/(x^2-y^2,xy,yx)$, are isomorphic when $\mathrm{char}(k) \neq 2$.. Consequently, in this case, the associated ribbon graph $\Gamma$ can be either a loop or a single edge.
	\end{proof}

\section{Repetitive algebras and $r$-fold trivial extensions of gentle algebras}\label{sec:3}

We begin by recall some basic definitions in representation theory.

\begin{definition}
	Let $A$ be a finite-dimensional $k$-algebra, and let $D=\mathrm{Hom}_k(-,k)$.
\begin{enumerate}
	\item The {\it repetitive algebra} $\widehat{A}$ of $A$ is the locally finite-dimensional algebra
\[
\widehat{A}=\bigoplus_{i\in \mathbb{Z}} A_i \;\oplus\; \bigoplus_{i\in \mathbb{Z}} DA_i,
\]
where each $A_i \cong A$. An element of $\widehat{A}$ can be written as a sequence $(a_i,f_i)_{i\in\mathbb{Z}}$ with only finitely many nonzero terms, where $a_i\in A_i$ and $f_i\in DA_i$. The multiplication is given by
\[
(a_i,f_i)\cdot (a_i',f_i') = ( a_i a_i',a_i f_i' + f_i a_{i+1}'),
\]
 for each $i\in\mathbb{Z}$.
There is a natural automorphism $\nu$ of $\widehat{A}$, called the {\it Nakayama automorphism}, 
given by the shift $\nu(a_i)=a_{i+1}$ and $\nu(f_i)=f_{i+1}$.

\item For a positive integer $r$, the {\it $r$-fold trivial extension} of $A$ is defined as the orbit algebra $$T_r(A):=\widehat{A}/\langle \nu^r \rangle.$$
In particular, when $r=1$, one has $T(A):=T_1(A) \cong \widehat{A}/\langle \nu \rangle$ is called the {\it trivial extension} of $A$.
\end{enumerate}
\end{definition}

\begin{definition}
	Let $A \cong kQ/I$ be a special biserial algebra.
We say that $A$ is a {\it gentle algebra} if, in addition, the following two conditions hold:
\begin{enumerate}
    \item The ideal $I$ is generated by paths of length $2$.
    
    \item For each arrow $\alpha \in Q_1$, there is at most one arrow $\beta \in Q_1$ and at most one arrow $\gamma \in Q_1$ such that 
    $\alpha\beta \in I$ and $\gamma\alpha \in I$.
\end{enumerate}
\end{definition}

Since a gentle algebra $A=kQ/I$ is monomial, there exists a set $\mathcal{M}$ of maximal paths in $A$, that is, for each $p \in \mathcal{M}$ and each arrow $\alpha \in Q_1$, one has $\alpha p = 0 = p \alpha$ in $A$. The set $\mathcal{M}$ is uniquely determined by $A = kQ/I$, since, by \cite[Proposition~2.5]{Green}, the ideal $I$ admits a unique minimal generating set $\mathcal{G}$ consisting of paths. Then there exists a ribbon graph $\Gamma_A$ (see \cite{Sch,OPS}) whose vertices correspond to 
\[
\overline{\mathcal{M}} := \mathcal{M} \cup \left\{\, i \in Q_0 \ \middle|\ 
\begin{array}{l}
i \text{ is a source with a single arrow starting at } i, \\
\text{or } i \text{ is a sink with a single arrow ending at } i, \\
\text{or there is a single arrow } \alpha \text{ ending at } i \\
\text{and a single arrow } \beta \text{ starting at } i \text{ with } \beta\alpha \notin I
\end{array}
\right\},
\]
and whose edges correspond to the vertices of $Q$. An edge $e$ is incident to a vertex $\vv \in \overline{\mathcal{M}}$ in $\Gamma_A$ if the corresponding path $p_\vv$ passes through $e \in Q_0$. The cyclic order around each vertex is naturally induced by the corresponding path.

It is well known that the BGA associated with $(\Gamma_A,{val})$ is the trivial extension of $A$  \cite{Sch}. Moreover, a {\it cutting set} $D$ of a BGA is defined by choosing exactly one arrow from each special cycle (up to cyclic permutation) $C_\alpha$ (see Subsection~\ref{subsec:def-fms-BGA}). Furthermore, each gentle algebra corresponds bijectively to a BGA with BG $(\Gamma, {val})$, together with a cutting set $D$ (see for example in \cite[Corollary 5.5]{GS}).

We omit some basic material on covering theory and refer the reader to \cite{BG1982,MP,Erd}. Now we begin by constructing the $r$-covering of a BGA with multiplicity one. Let $(\Gamma = (V, H, s, \iota, \rho), d=val)$ be a BG with multiplicity one, meaning that for every vertex $\vv \in V(\Gamma)$, the value $d(\vv)={val}(\vv)$. Denote by $A = kQ/I$ the Brauer graph algebra associated with $(\Gamma,val)$. 
For any cyclic $\langle \rho \rangle$-orbit in $H$, choose an angle from $h_n$ to $h_1$ (so that $\rho(h_n) = h_1$), and fix an ordering $\{h_1, \dots, h_n\}$ satisfying
\begin{itemize}
	\item $\rho(h_i) = h_{i+1} \quad (1 \le i \le n-1)$;
	\item $h_i\neq h_j$, for all $i\neq j$.
\end{itemize}
Such a selection of angles determines a cutting set $D$ of the quiver $Q$. We call a $\langle \rho \rangle$-orbit $\{h_1, \dots, h_n\}$ a $\langle \rho \rangle$-orbit {\it with respect to $D$}.
We define two new ribbon graphs associated with $\Gamma$ and $D$.

\medskip

\noindent
(1) The infinite covering graph $\Gamma_D^{\mathbb{Z}}$ is defined by
\[
\Gamma_D^{\mathbb{Z}} = \bigl(V, \bigsqcup_{j \in \mathbb{Z}} H^{(j)}, s, \iota, \rho \bigr),
\]
where $H^{(j)} = H$ for each $j \in \mathbb{Z}$, and elements of $H^{(j)}$ are denoted by $h^{(j)}$ for $h \in H$.

The structure maps are given by
\[
s(h^{(j)}) = s(h), \qquad \iota(h^{(j)}) = \iota(h)^{(j)}.
\]
For each $\langle \rho \rangle$-orbit $\{h_1, \dots, h_n\}$ with respect to $D$ in $H$, define the permutation $\rho$ on $\bigsqcup_{j \in \mathbb{Z}} H^{(j)}$ by
\[
\rho(h_i^{(j)}) = h_{i+1}^{(j)} \quad (1 \le i \le n-1), 
\qquad 
\rho(h_n^{(j)}) = h_1^{(j+1)},
\]
for all $j \in \mathbb{Z}$.

\medskip

\noindent
(2) For a positive integer $r$, the finite graph $\Gamma_D^{(r)}$ is defined as
\[
\Gamma_D^{(r)} = \bigl(V, \bigsqcup_{j \in \mathbb{Z}/r\mathbb{Z}} H^{(j)}, s, \iota, \rho \bigr),
\]
where $H^{(j)} = H$ for each $j \in \mathbb{Z}/r\mathbb{Z}$.

The maps $s$, $\iota$, and $\rho$ are defined in the same way as above, namely,
\[
s(h^{(j)}) = s(h), \qquad \iota(h^{(j)}) = \iota(h)^{(j)},
\]
and for each $\langle \rho \rangle$-orbit $\{h_1, \dots, h_n\}$ with respect to $D$ in $H$,
\[
\rho(h_i^{(j)}) = h_{i+1}^{(j)} \quad (1 \le i \le n-1), 
\qquad 
\rho(h_n^{(j)}) = h_1^{(j+1)},
\]
for all $j \in \mathbb{Z}/r\mathbb{Z}$.

Then we have the following proposition.

\begin{proposition}\label{prop:r-covering}
	Let $(\Gamma, d)$ be a BG, $D$ be a cutting set, and $r$ be an arbitrary positive integer. Then:  
\begin{enumerate}
    \item $(\Gamma_D^{(r)}, d)$ (resp. $(\Gamma_D^{\mathbb{Z}}, d)$) is an AFBG.
    \item Let $A$ (resp.\ $B$) be the BGA (resp.\ AFBGA) associated with $(\Gamma, d)$ (resp.\ $(\Gamma_D^{\mathbb{Z}}, d)$). Then $B$ is an $\mathbb{Z}$-covering of $A$.
    \item Let $A$ (resp.\ $B$) be the BGA (resp.\ AFBGA) associated with $(\Gamma, d)$ (resp.\ $(\Gamma_D^{(r)}, d)$). Then $B$ is an $r$-covering of $A$.
\end{enumerate}
\end{proposition}
\begin{proof}
By the construction of $\Gamma_D^{(r)}$ (resp.\ $\Gamma_D^{\mathbb{Z}}$), it naturally forms a ribbon graph. It is indeed an AFBG, since the Nakayama automorphism $\nu = \rho^{d}$ sends each $h^{(j)}$ to $h^{(j+1)}$, and correspondingly sends $\iota(h^{(j)}) = \iota(h)^{(j)}$ to $\iota(h^{(j+1)}) = \iota(h)^{(j+1)}$ for every $j \in \mathbb{Z}/r\mathbb{Z}$ (resp.\ $j \in \mathbb{Z}$). The action of the Nakayama permutation is admissible because each $\langle \sigma \rangle$-orbit is of the form $\{h^{(1)}, \ldots, h^{(r)}, \ldots\}$, and none of the $\iota(h^{(j)})$ lie in the same orbit as $h^{(j)}$.

Consider the morphisms of ribbon graphs
$$
f : \Gamma_D^{\mathbb{Z}} \longrightarrow \Gamma, \qquad h^{(j)} \longmapsto h,
$$
and
$$
f : \Gamma_D^{(r)} \longrightarrow \Gamma, \qquad h^{(j)} \longmapsto h.
$$
By the definitions of $\Gamma_D^{\mathbb{Z}}$, $\Gamma_D^{(r)}$, and \cite[Definition~3.1]{LL2}, these maps are natural coverings of the corresponding 
AFBGs. Consequently, by \cite[Theorem~5.8]{LL2}, the associated quiver algebras also form coverings. These coverings are $\mathbb{Z}$-coverings and $r$-coverings, respectively, since in the first case there is a $\mathbb{Z}$-action on the index set $\mathbb{Z}$, and in the second case each $\langle \nu \rangle$-orbit contains exactly $r$ elements.
\end{proof}

Now we show the main theorem of this section.

\begin{theorem}\label{thm:r-fold-tri}
	Let $A=kQ/I$ be a gentle algebra and $((\Gamma,d),D)$ be the BG associated with a cutting set corresponding to the trivial extension $T(A)$ of $A$. Then the AFBGA corresponding to $(\Gamma^\mathbb{Z}_D,d)$ is isomorphic to the repetitive algebra $\hat{A}$, and the AFBGA corresponding to $(\Gamma^{(r)}_D,d)$ is isomorphic to the $r$-fold trivial extension $T_r(A)$.
\end{theorem}

\begin{proof}
	This is natural, since the quiver and defining relations of the AFBGA associated with $(\Gamma_D^{\mathbb{Z}}, d)$ coincide with those of $\widehat{A}$ constructed in \cite{Rin2}. 
Moreover, the AFBGA associated with $(\Gamma_D^{(r)}, d)$ is isomorphic to the $r$-fold trivial extension $T_r(A)$. Indeed, $\Gamma_D^{(r)} \cong \Gamma_D^{\mathbb{Z}} / \langle \nu^r \rangle$, and hence the corresponding AFBGA is isomorphic to $\widehat{A} / \langle \nu^r \rangle$.
\end{proof}

We end this section with some examples.

	\begin{example}\label{exa:2-fold-extension}
		Consider the BGA $\Lambda$ with associated BG $(\Gamma,d=2)$ where $\Gamma$ is given by

		\begin{center}

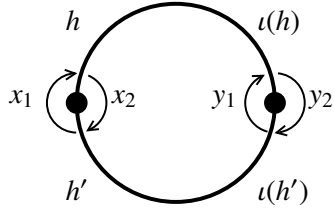
\tikzset{every picture/.style={line width=0.75pt}} 

\begin{tikzpicture}[x=0.75pt,y=0.75pt,yscale=-1,xscale=1]

\draw  [fill={rgb, 255:red, 0; green, 0; blue, 0 }  ,fill opacity=1 ] (252,98) .. controls (252,95.24) and (254.24,93) .. (257,93) .. controls (259.76,93) and (262,95.24) .. (262,98) .. controls (262,100.76) and (259.76,103) .. (257,103) .. controls (254.24,103) and (252,100.76) .. (252,98) -- cycle ;
\draw  [line width=1.5]  (157,98) .. controls (157,70.39) and (179.39,48) .. (207,48) .. controls (234.61,48) and (257,70.39) .. (257,98) .. controls (257,125.61) and (234.61,148) .. (207,148) .. controls (179.39,148) and (157,125.61) .. (157,98) -- cycle ;
\draw  [fill={rgb, 255:red, 0; green, 0; blue, 0 }  ,fill opacity=1 ] (152,98) .. controls (152,95.24) and (154.24,93) .. (157,93) .. controls (159.76,93) and (162,95.24) .. (162,98) .. controls (162,100.76) and (159.76,103) .. (157,103) .. controls (154.24,103) and (152,100.76) .. (152,98) -- cycle ;
\draw  [draw opacity=0] (156.76,113) .. controls (148.59,112.87) and (142,106.2) .. (142,98) .. controls (142,89.96) and (148.33,83.4) .. (156.28,83.02) -- (157,98) -- cycle ; \draw   (156.76,113) .. controls (148.59,112.87) and (142,106.2) .. (142,98) .. controls (142,89.96) and (148.33,83.4) .. (156.28,83.02) ;  
\draw  [draw opacity=0] (162.83,84.18) .. controls (168.22,86.45) and (172,91.78) .. (172,98) .. controls (172,104.02) and (168.45,109.22) .. (163.33,111.6) -- (157,98) -- cycle ; \draw   (162.83,84.18) .. controls (168.22,86.45) and (172,91.78) .. (172,98) .. controls (172,104.02) and (168.45,109.22) .. (163.33,111.6) ;  
\draw   (150.96,81.32) -- (156.23,83.19) -- (151.57,86.29) ;
\draw   (168.21,111.24) -- (162.65,111.86) -- (165.49,107.04) ;
\draw  [draw opacity=0] (253.24,112.53) .. controls (246.78,110.86) and (242,104.99) .. (242,98) .. controls (242,91.83) and (245.73,86.52) .. (251.06,84.22) -- (257,98) -- cycle ; \draw   (253.24,112.53) .. controls (246.78,110.86) and (242,104.99) .. (242,98) .. controls (242,91.83) and (245.73,86.52) .. (251.06,84.22) ;  
\draw   (245.36,84.49) -- (250.95,84.31) -- (247.74,88.89) ;
\draw  [draw opacity=0] (258.54,83.08) .. controls (266.1,83.85) and (272,90.23) .. (272,98) .. controls (272,105.55) and (266.42,111.8) .. (259.16,112.85) -- (257,98) -- cycle ; \draw   (258.54,83.08) .. controls (266.1,83.85) and (272,90.23) .. (272,98) .. controls (272,105.55) and (266.42,111.8) .. (259.16,112.85) ;  
\draw   (263.22,114.41) -- (257.81,113.03) -- (262.16,109.53) ;

\draw (121,90) node [anchor=north west][inner sep=0.75pt]   [align=left] {$x_1$};
\draw (172.75,90) node [anchor=north west][inner sep=0.75pt]   [align=left] {$x_2$};
\draw (224,90) node [anchor=north west][inner sep=0.75pt]   [align=left] {$y_1$};
\draw (272.75,90) node [anchor=north west][inner sep=0.75pt]   [align=left] {$y_2$};
\draw (150,50) node [anchor=north west][inner sep=0.75pt]   [align=left] {$h$};
\draw (150,135) node [anchor=north west][inner sep=0.75pt]   [align=left] {$h'$};
\draw (247,50) node [anchor=north west][inner sep=0.75pt]   [align=left] {$\iota(h)$};
\draw (247,135) node [anchor=north west][inner sep=0.75pt]   [align=left] {$\iota(h')$};

\end{tikzpicture}

		\end{center}
		Then $\Lambda=kQ_\Gamma/I_\Lambda$, where $$I_\Lambda=\langle x_1x_2-y_1y_2,x_2x_1-y_2y_1,x_1y_2,y_2x_1,x_2y_1,y_1x_2\rangle$$ 
		with the quiver $Q_\Gamma$
	\[\begin{tikzcd}
1 \arrow[r, "x_1", shift left=4] \arrow[r, "y_1" description, shift right=2] & 2 \arrow[l, "x_2" description, shift right] \arrow[l, "y_2", shift left=5]
\end{tikzcd}\]	
Note that the vertices $1$ and $2$ in $Q_\Gamma$ corresponding to the edges $\bar{h}$ and $\bar{h'}$ in $\Gamma$.

		Give different cutting sets $D_1=\{x_2,y_2\}$ and $D_2=\{x_2,y_1\}$. 
		The gentle algebra associated with $D_1$ is $A=kQ_1$, the Kronecker algebra, and the gentle algebra associated with $D_2$ is $A'=kQ_2/\langle x_1y_2,y_2x_1\rangle$. 
		\[\begin{tikzcd}
			{Q_1:} & 1 & 2 && {Q_2:} & 1 & 2
			\arrow["{x_1}", shift left, from=1-2, to=1-3]
			\arrow["{y_1}"', shift right, from=1-2, to=1-3]
			\arrow["{x_1}", shift left, from=1-6, to=1-7]
			\arrow["{y_2}", shift left, from=1-7, to=1-6]
		\end{tikzcd}\]

	Now consider $\Gamma_{D_1}^{\mathbb{Z}}$ (on the left) and $\Gamma_{D_1}^{(2)}$ (on the right), as shown below.

		\begin{center}

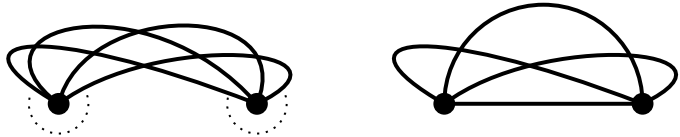
\tikzset{every picture/.style={line width=0.75pt}} 

\begin{tikzpicture}[x=0.75pt,y=0.75pt,yscale=-1,xscale=1]

\draw  [fill={rgb, 255:red, 0; green, 0; blue, 0 }  ,fill opacity=1 ] (252,98) .. controls (252,95.24) and (254.24,93) .. (257,93) .. controls (259.76,93) and (262,95.24) .. (262,98) .. controls (262,100.76) and (259.76,103) .. (257,103) .. controls (254.24,103) and (252,100.76) .. (252,98) -- cycle ;
\draw  [fill={rgb, 255:red, 0; green, 0; blue, 0 }  ,fill opacity=1 ] (152,98) .. controls (152,95.24) and (154.24,93) .. (157,93) .. controls (159.76,93) and (162,95.24) .. (162,98) .. controls (162,100.76) and (159.76,103) .. (157,103) .. controls (154.24,103) and (152,100.76) .. (152,98) -- cycle ;
\draw [line width=1.5]    (157,98) .. controls (98.75,64.5) and (144.25,54.5) .. (257,98) ;
\draw [line width=1.5]    (157,98) .. controls (105.75,59) and (201.25,34.5) .. (257,98) ;
\draw [line width=1.5]    (157,98) .. controls (169.75,45.5) and (278.75,45) .. (257,98) ;
\draw [line width=1.5]    (157,98) .. controls (206.25,61) and (315.75,69.5) .. (257,98) ;
\draw  [draw opacity=0][dash pattern={on 0.84pt off 2.51pt}] (171.36,93.64) .. controls (171.77,95.02) and (172,96.48) .. (172,98) .. controls (172,106.28) and (165.28,113) .. (157,113) .. controls (148.72,113) and (142,106.28) .. (142,98) .. controls (142,96.48) and (142.23,95.01) .. (142.65,93.63) -- (157,98) -- cycle ; \draw  [dash pattern={on 0.84pt off 2.51pt}] (171.36,93.64) .. controls (171.77,95.02) and (172,96.48) .. (172,98) .. controls (172,106.28) and (165.28,113) .. (157,113) .. controls (148.72,113) and (142,106.28) .. (142,98) .. controls (142,96.48) and (142.23,95.01) .. (142.65,93.63) ;  
\draw  [draw opacity=0][dash pattern={on 0.84pt off 2.51pt}] (271.36,93.64) .. controls (271.77,95.02) and (272,96.48) .. (272,98) .. controls (272,106.28) and (265.28,113) .. (257,113) .. controls (248.72,113) and (242,106.28) .. (242,98) .. controls (242,96.1) and (242.35,94.27) .. (243,92.6) -- (257,98) -- cycle ; \draw  [dash pattern={on 0.84pt off 2.51pt}] (271.36,93.64) .. controls (271.77,95.02) and (272,96.48) .. (272,98) .. controls (272,106.28) and (265.28,113) .. (257,113) .. controls (248.72,113) and (242,106.28) .. (242,98) .. controls (242,96.1) and (242.35,94.27) .. (243,92.6) ;  
\draw  [fill={rgb, 255:red, 0; green, 0; blue, 0 }  ,fill opacity=1 ] (446.5,97.96) .. controls (446.5,95.2) and (448.74,92.96) .. (451.5,92.96) .. controls (454.26,92.96) and (456.5,95.2) .. (456.5,97.96) .. controls (456.5,100.72) and (454.26,102.96) .. (451.5,102.96) .. controls (448.74,102.96) and (446.5,100.72) .. (446.5,97.96) -- cycle ;
\draw  [fill={rgb, 255:red, 0; green, 0; blue, 0 }  ,fill opacity=1 ] (346.5,97.96) .. controls (346.5,95.2) and (348.74,92.96) .. (351.5,92.96) .. controls (354.26,92.96) and (356.5,95.2) .. (356.5,97.96) .. controls (356.5,100.72) and (354.26,102.96) .. (351.5,102.96) .. controls (348.74,102.96) and (346.5,100.72) .. (346.5,97.96) -- cycle ;
\draw [line width=1.5]    (351.5,97.96) .. controls (293.25,64.46) and (338.75,54.46) .. (451.5,97.96) ;
\draw [line width=1.5]    (351.5,97.96) .. controls (400.75,60.96) and (510.25,69.46) .. (451.5,97.96) ;
\draw [line width=1.5]    (351.5,97.96) -- (451.5,97.96) ;
\draw  [draw opacity=0][line width=1.5]  (351.5,98) .. controls (351.5,97.99) and (351.5,97.97) .. (351.5,97.96) .. controls (351.5,70.35) and (373.89,47.96) .. (401.5,47.96) .. controls (428.82,47.96) and (451.01,69.86) .. (451.49,97.06) -- (401.5,97.96) -- cycle ; \draw  [line width=1.5]  (351.5,98) .. controls (351.5,97.99) and (351.5,97.97) .. (351.5,97.96) .. controls (351.5,70.35) and (373.89,47.96) .. (401.5,47.96) .. controls (428.82,47.96) and (451.01,69.86) .. (451.49,97.06) ;

\end{tikzpicture}

		\end{center}
		Then, by Theorem \ref{thm:r-fold-tri}, the repetitive algebra $\hat{A} = kQ_{\Gamma_{D_1}^{\mathbb{Z}}}/I_\Lambda$ is the AFBGA associated with $(\Gamma_{D_1}^{\mathbb{Z}}, d=2)$, where the quiver $Q_{\Gamma_{D_1}^{\mathbb{Z}}}$ is given as follows.
$$
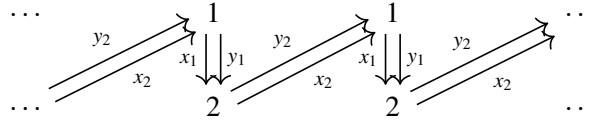
\begin{tikzcd}
\cdots                                                                 &  & 1 \arrow[d, "x_1"', shift right] \arrow[d, "y_1", shift left]     &  & 1 \arrow[d, "y_1", shift left] \arrow[d, "x_1"', shift right]     &  & \cdots \\
\cdots \arrow[rru, "x_2"', shift right] \arrow[rru, "y_2", shift left] &  & 2 \arrow[rru, "y_2", shift left] \arrow[rru, "x_2"', shift right] &  & 2 \arrow[rru, "x_2"', shift right] \arrow[rru, "y_2", shift left] &  & \cdots
\end{tikzcd}$$
And the $2$-fold trivial extension $T_2(A) = kQ_{\Gamma_{D_1}^{(2)}}/I_\Lambda$ is the AFBGA associated with $(\Gamma_{D_1}^{(2)}, d=2)$, where the quiver $Q_{\Gamma_{D_1}^{(2)}}$ is given as follows.
$$
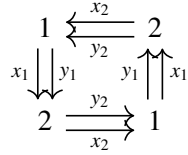
\begin{tikzcd}
1 \arrow[d, "x_1"', shift right] \arrow[d, "y_1", shift left] & 2 \arrow[l, "y_2", shift left] \arrow[l, "x_2"', shift right] \\
2 \arrow[r, "y_2", shift left] \arrow[r, "x_2"', shift right] & 1 \arrow[u, "x_1"', shift right] \arrow[u, "y_1", shift left]
\end{tikzcd}$$

Meanwhile, consider $\Gamma_{D_2}^{\mathbb{Z}}$ (on the left) and $\Gamma_{D_2}^{(2)}$ (on the right), as shown below.
		\begin{center}

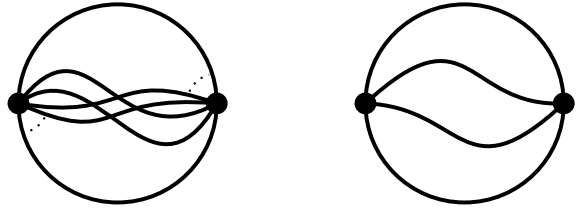
\tikzset{every picture/.style={line width=0.75pt}} 

\begin{tikzpicture}[x=0.75pt,y=0.75pt,yscale=-1,xscale=1]

\draw  [fill={rgb, 255:red, 0; green, 0; blue, 0 }  ,fill opacity=1 ] (100,108) .. controls (100,105.24) and (102.24,103) .. (105,103) .. controls (107.76,103) and (110,105.24) .. (110,108) .. controls (110,110.76) and (107.76,113) .. (105,113) .. controls (102.24,113) and (100,110.76) .. (100,108) -- cycle ;
\draw  [fill={rgb, 255:red, 0; green, 0; blue, 0 }  ,fill opacity=1 ] (200,108) .. controls (200,105.24) and (202.24,103) .. (205,103) .. controls (207.76,103) and (210,105.24) .. (210,108) .. controls (210,110.76) and (207.76,113) .. (205,113) .. controls (202.24,113) and (200,110.76) .. (200,108) -- cycle ;
\draw  [line width=1.5]  (105,108) .. controls (105,80.39) and (127.39,58) .. (155,58) .. controls (182.61,58) and (205,80.39) .. (205,108) .. controls (205,135.61) and (182.61,158) .. (155,158) .. controls (127.39,158) and (105,135.61) .. (105,108) -- cycle ;
\draw [line width=1.5]    (105,108) .. controls (142,59.75) and (152.5,136.75) .. (205,108) ;
\draw [line width=1.5]    (105,108) .. controls (145,78) and (171,166.75) .. (205,108) ;
\draw [line width=1.5]    (105,108) .. controls (161.5,117.75) and (151.5,89.25) .. (205,108) ;
\draw [line width=1.5]    (105,108) .. controls (157,131.25) and (143,102.75) .. (205,108) ;
\draw  [draw opacity=0][dash pattern={on 0.84pt off 2.51pt}] (118.23,115.07) .. controls (115.9,119.43) and (111.47,122.5) .. (106.3,122.94) -- (105,108) -- cycle ; \draw  [dash pattern={on 0.84pt off 2.51pt}] (118.23,115.07) .. controls (115.9,119.43) and (111.47,122.5) .. (106.3,122.94) ;  
\draw  [draw opacity=0][dash pattern={on 0.84pt off 2.51pt}] (191.2,102.11) .. controls (193,97.9) and (196.67,94.68) .. (201.17,93.49) -- (205,108) -- cycle ; \draw  [dash pattern={on 0.84pt off 2.51pt}] (191.2,102.11) .. controls (193,97.9) and (196.67,94.68) .. (201.17,93.49) ;  
\draw  [fill={rgb, 255:red, 0; green, 0; blue, 0 }  ,fill opacity=1 ] (275,108) .. controls (275,105.24) and (277.24,103) .. (280,103) .. controls (282.76,103) and (285,105.24) .. (285,108) .. controls (285,110.76) and (282.76,113) .. (280,113) .. controls (277.24,113) and (275,110.76) .. (275,108) -- cycle ;
\draw  [fill={rgb, 255:red, 0; green, 0; blue, 0 }  ,fill opacity=1 ] (375,108) .. controls (375,105.24) and (377.24,103) .. (380,103) .. controls (382.76,103) and (385,105.24) .. (385,108) .. controls (385,110.76) and (382.76,113) .. (380,113) .. controls (377.24,113) and (375,110.76) .. (375,108) -- cycle ;
\draw  [line width=1.5]  (280,108) .. controls (280,80.39) and (302.39,58) .. (330,58) .. controls (357.61,58) and (380,80.39) .. (380,108) .. controls (380,135.61) and (357.61,158) .. (330,158) .. controls (302.39,158) and (280,135.61) .. (280,108) -- cycle ;
\draw [line width=1.5]    (280,108) .. controls (333.25,58.75) and (330.25,110.25) .. (380,108) ;
\draw [line width=1.5]    (280,108) .. controls (329.75,109.75) and (328.75,155.75) .. (380,108) ;

\end{tikzpicture}

		\end{center}
		Then, by Theorem \ref{thm:r-fold-tri}, the repetitive algebra $\hat{A'} = kQ_{\Gamma_{D_2}^{\mathbb{Z}}}/I_\Lambda$ is the AFBGA associated with $(\Gamma_{D_2}^{\mathbb{Z}}, d=2)$, where the quiver $Q_{\Gamma_{D_2}^{\mathbb{Z}}}$ is given as follows.
$$
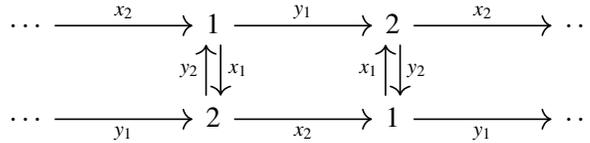
\begin{tikzcd}
\cdots \arrow[rr, "x_2"]  &  & 1 \arrow[d, "x_1", shift left] \arrow[rr, "y_1"]  &  & 2 \arrow[d, "y_2", shift left] \arrow[rr, "x_2"]  &  & \cdots \\
\cdots \arrow[rr, "y_1"'] &  & 2 \arrow[u, "y_2", shift left] \arrow[rr, "x_2"'] &  & 1 \arrow[u, "x_1", shift left] \arrow[rr, "y_1"'] &  & \cdots
\end{tikzcd}$$
And the $2$-fold trivial extension $T_2(A') = kQ_{\Gamma_{D_2}^{(2)}}/I_\Lambda$ is the AFBGA associated with $(\Gamma_{D_2}^{(2)}, d=2)$, where the quiver $Q_{\Gamma_{D_2}^{(2)}}$ is given as follows.
$$
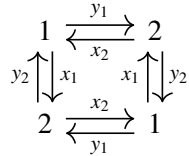
\begin{tikzcd}
1 \arrow[d, "x_1", shift left] \arrow[r, "y_1", shift left] & 2 \arrow[l, "x_2", shift left] \arrow[d, "y_2", shift left] \\
2 \arrow[u, "y_2", shift left] \arrow[r, "x_2", shift left] & 1 \arrow[u, "x_1", shift left] \arrow[l, "y_1", shift left]
\end{tikzcd}$$
	\end{example}

	\section{Invariants under derived equivalences}\label{sec:4}

	In the case of BGAs, \cite{AZ} and \cite{OZ} provide a complete derived invariant characterizing derived equivalences between BGAs. In this section, we study invariants under derived equivalences for AFBGAs. The situation is analogous to that of BGAs. 
	
	By \cite[Corollary 2.13]{RZ}, two local algebras are derived equivalent if and only if they are Morita equivalent. By Proposition \ref{prop:determined-by-BG}, AFBGAs are completely determined by their associated AFBGs, except for two exceptional local cases. Therefore, we restrict our discussion to the non-local case. We also remark that the following theorem also holds for infinite-dimensional AFBGAs.

	\begin{theorem}\label{thm:derived-eqv-between-BGAs}
 Let $A$ and $B$ be non-local AFBGAs associated with AFBGs $(\Gamma, d)$ and $(\Gamma', d')$, respectively. If $A$ and $B$ are derived equivalent, then the following statements hold:
\begin{enumerate}
\item The associated BGAs $A_{\mathrm{red}}$ and $B_{\mathrm{red}}$ are derived equivalent;
\item The following conditions are satisfied:
\begin{itemize}
\item $\Gamma$ and $\Gamma'$ have the same number of vertices and edges;
\item $(\Gamma,d)$ and $(\Gamma',d')$ have the same multisets of vertex multiplicities;
\item either both $\Gamma$ and $\Gamma'$ are bipartite, or neither is.
\end{itemize}
\end{enumerate}
\end{theorem}

\begin{proof}
Denote by $\nu_A$ (resp.~$\nu_B$) the Nakayama automorphism of $A$ (resp.~$B$). By \cite{GR,Asa}, it is natural to regard such locally bounded quiver algebras as spectroids. Since $A$ and $B$ are derived equivalent, it follows from \cite[Proposition~6.3]{Ric} that there exists a tilting spectroid $E$ for $A$ such that $E$ is equivalent to $B$ as categories. As both $A$ and $B$ are self-injective (Proposition~\ref{prop:bfBGA-basic}), it follows from \cite[Theorem~2.1]{AR} and \cite[Theorem~A.4]{Ai} that $E$ is $\nu_A$-stable. By the definition of AFBGAs, both $\langle \nu_A \rangle$ and $\langle \nu_B \rangle$ are finite cyclic groups (or both isomorphic to $\mathbb{Z}$). Moreover, since derived equivalent self-injective algebras have conjugate Nakayama permutations (\cite[Theorem~4.1]{XZ}), we obtain $\langle \nu_A \rangle \cong \langle \nu_B \rangle$. Therefore, by \cite[Theorem~3.5]{Asa}, the reduced algebra $A_{\mathrm{red}} \cong A/\langle \nu_A \rangle$ is derived equivalent to $B_{\mathrm{red}} \cong B/\langle \nu_B \rangle$.

For an AFBG $(\Gamma,d)$, the ribbon graphs $\Gamma$ and $\Gamma_{\mathrm{red}}$ have the same set of vertices. Moreover, by~\cite[Proposition 4.5]{AZ}, the reduced ribbon graphs $\Gamma_{\mathrm{red}}$ and $\Gamma'_{\mathrm{red}}$ have the same number of vertices. It follows that $\Gamma$ and $\Gamma'$ have the same number of vertices.
Since the number of edges of a ribbon graph corresponds to the number of isomorphism classes of simple modules of the associated algebra, it follows that $\Gamma$ and $\Gamma'$ have the same number of edges.
Assume that the (finite) cyclic group generated by the Nakayama automorphism has order $r$, and that the multiset of vertex multiplicities of $\Gamma_{\mathrm{red}}$ is $\{m_1,\dots,m_n\}$. Then the multiset of vertex multiplicities of $\Gamma$ is given by $\{m_1/r,\dots,m_n/r\}$. Hence, by~\cite[Proposition 4.5]{AZ}, $\Gamma$ and $\Gamma'$ have the same multiset of vertex multiplicities.
Finally, recall that a graph is bipartite if and only if it contains no odd cycles (\cite[Theorem 4.7]{BM}). If $\Gamma_{\mathrm{red}}$ is not bipartite, then since $\Gamma \to \Gamma_{\mathrm{red}}$ is a covering of ribbon graphs (\cite{GSS,LL2}), any odd cycle in $\Gamma_{\mathrm{red}}$ lifts to an odd cycle in $\Gamma$. Hence $\Gamma$ is not bipartite. By~\cite{A,OZ}, it follows that either both or neither of $\Gamma$ and $\Gamma'$ are bipartite.
\end{proof}

We end this section with the following remarks.

\begin{remark}\label{rmk:1}
	Note that the invariants in the above theorem are not complete. For instance, in Example \ref{exa:2-fold-extension}, the algebras $T_2(A)$ and $T_2(A')$ satisfy all the invariants listed above. However, by~\cite{BR,ES}, $T_2(A)$ is $4$-domestic whereas $T_2(A')$ is $2$-domestic. Hence, $T_2(A)$ and $T_2(A')$ are not stably equivalent, and therefore, by~\cite{Ric2}, they are not derived equivalent.
\end{remark}

\begin{remark}\label{rmk:2}
	We expect that a complete set of derived invariants for AFBGAs should extend Theorem~\ref{thm:derived-eqv-between-BGAs} by incorporating information on the faces of the ribbon graph and their perimeters. 
Moreover, by~\cite[Proposition 5.3]{LL3}, for derived equivalent AFBGAs, the $\nu^{-1}(\rho\iota)^2$-orbits on $\Gamma$ and $\Gamma'$ also coincide. However, this invariant still fails to distinguish the two non-derived equivalent algebras in Remark~\ref{rmk:1}. 

Therefore, unlike in the case of BGAs, the existence of a nontrivial Nakayama automorphism $\nu$ makes the situation more subtle. Based on existing experience, a more refined analysis of the relationship between the tubes in the stable Auslander--Reiten quiver of AFBGAs and the faces of the associated ribbon graph will be required in future work.
\end{remark}

\end{document}